\documentclass[11pt]{amsart}
\usepackage{bm}
\usepackage{fullpage}
\usepackage{amssymb}
\usepackage{amsmath,amsfonts,amsthm}
\usepackage{hyperref}
\usepackage{color}
\usepackage{cite}

\newtheorem{alphtheorem}{Theorem}

\newtheorem{theorem}{Theorem}
\newtheorem{problem}{Problem}[theorem]

\newtheorem{lemma}[theorem]{Lemma}

\newtheorem{definition}[theorem]{Definition}
\newtheorem{claim}[theorem]{Claim}

\newtheorem{corollary}[theorem]{Corollary}
\newtheorem{conjecture}{Conjecture}

\DeclareMathOperator\G{\mathcal{G}}

\DeclareMathOperator\R{\mathbb{R}}

\title{ON GENERALIZED LIST $\G$-FREE COLOURINGS OF GRAPHS}

\author{Yaser Rowshan$^1$}
\keywords{$\G$-free Subset, $k$-choosable, $d$-regular graphs, Conditional Coloring, Vertex Arboricity,  $L$-$G$-free-colorable.}
\subjclass[2010]{05C15.}
\address{$^1$Y. Rowshan, 
	Department of Mathematics, Institute for Advanced Studies in Basic Sciences (IASBS), Zanjan 45137-66731, Iran.}
\email{y.rowshan@iasbs.ac.ir,~~~y.rowshan.math@gmail.com.}
 
\begin{document}
	\maketitle 	
	\begin{abstract} 
 For given graph $H$ and graphical property $P$, the conditional chromatic number $\chi(H,P)$   of $H$, is the smallest number $k$, so that   $V(H)$ can be decomposed into  sets $V_1,V_2,\ldots, V_k$, in which $H[V_i]$ satisfies the property $P$, for each $1\leq i\leq k$. When property $P$ be  that  each color class contains no copy of $G$, we write $\chi_{G}(H)$ instead of  $\chi(G,P)$, which is called the $G$-free chromatic number. Due to this, we say  $H$ has a $k$-$G$-free coloring if there is a map $c : V(H) \longrightarrow \{1,\ldots,k\}$, so that each of the  color classes of $c$ be $G$-free. Assume that for each vertex $v$ of a graph $H$ is assigned a set $L(V)$ of colors, called a color list. Set $g(L) = \{g(v): v\in V(H)\}$, that is the set of colors chosen for the vertices of $H$ under $g$. An $L$-coloring $g$ is called  $G$-free,  so that:
\begin{itemize}	
\item  $g(v)\in L(v)$,	 for any $v\in V(H)$.
\item  $ H[V_i]$ is $G$-free for each $i=1,2,\ldots, L$.
\end{itemize}
If there exists an $L$-coloring of $H$, then $H$ is called $L$-$G$-free-colorable. A graph $H$ is said to be $k$-$G$-free-choosable if there exists an $L$-coloring for any list-assignment $L$ satisfying  $|L(V)|\geq k$ for each $v\in V(H)$, and  $H[V_i]$ be $G$-free for each $i=1,2,\ldots, L$. Let  graph $H$ and  a collection of graphs $\G$ are given, the $\chi_{\G}^L(H)$ of $H$ is the last integer $k$, so that $H$ is $k$-$\G$-free-choosable i.e. $H[V_i]$ is $\G$-free for each $i=1,2,\ldots, k$ i.e. contains no copy of  any member of $\G$. In this article, we determin some upper bounds for $\chi_G^L(H)$, in term of the $\Delta(H), |V(H)|$ and $\delta(G)$. In particular, we show that $\chi_G^L(H)=\chi_G(H)$ for some graph $H$ and $G$,  $\chi_G^L(H\oplus H')\leq \chi_G^L(H)+\chi_G^L(H')$ for each $G$, $H$, and $H'$. Also, we show that  $\chi_{\G}(H\oplus K_n)=\chi^L_{\G}(H\oplus K_n)$, where $\G$ is a collection of all $d$-regular graphs, and some $n$.
	\end{abstract}
	
	\section{Introduction} 
	All graphs $G$ considered here are undirected, simple, and finite graphs. For  given graph  $H=(V(H),E(H))$, the  maximum degree of $H$ is denoted by  $\Delta(H)$, and the minimum degree of $H$ is denoted by $\delta(H)$. 
	 If $v$ be a vertex of $H$, the degree and neighbors of $v$ are denoted by $\deg_H{(v)}$ ($\deg{(v)}$) and $N_H(v)$, respectively.  
  The join of two graphs $G$ and $H$ is denoted by $G\oplus H$ and obtained from $G$ and $H$ by joining each vertex of $G$ to all vertices of $H$.
 We use  $\chi(H)$ to denote the  chromatic number of $H$. For given graph $H$, let each vertex $v$ of  $H$ be assigned a set $L(V)$ of colors, called a color list. An $L$-coloring of $H$ is a vertex-coloring $c$ so that:
 \begin{itemize}	
 	\item  For any $v\in V(H)$,	$c(v)\in L(v)$.
	\item  $c(v)\neq c(v')$ for each $vv'\in E(H)$.
\end{itemize}
 If there exists an $L$-coloring of $H$, then $H$ is called $L$-colorable. A graph $H$ is said to be $k$-choosable if there exists $L$-coloring for any list-assignment $L$ satisfying  $|L(V)|\geq k$ for each $v\in V(H)$. The choice number $\chi_L(H)$ of $H$ is the minimum integer $k$ so that $H$ is $k$-choosable. Note that $\chi(H)\leq \chi_L(H)$ for any graph $H$, however, equality does not necessarily hold.  The following particular case has been proved by  Ohba \cite{ohba2002chromatic}: 
\begin{alphtheorem}\label{thm1}{\rm\cite{ohba2002chromatic}}
	If $|V(H)|\leq  \chi(H)+\sqrt{2 \chi(H)}$, then:
	\[\chi_L(H)=\chi(H).\]   
\end{alphtheorem}
\begin{alphtheorem}\label{thm2}{\rm\cite{ohba2002chromatic}}
	If $|V(H)|+|V(G)|\leq  \chi(H)+\chi(G)+\sqrt{2 (\chi(H)+\chi(G))}$, then:
	\[\chi_L(H\oplus G)=\chi(H)+\chi(G).\]   
\end{alphtheorem}
One can refer to \cite{ kozik2014towards, kierstead2016choice, xu2018list, mudrock2018list} and \cite{zhu2021chromatic}  and their references for further studies about $L$-coloring of graph.
\subsection{Conditional Coloring} For given graph $H$ and graphical property $P$, the conditional chromatic number $\chi(H,P)$  of $H$, is the smallest number $k$, so that $V(H)$ can be decomposed into  sets $V_1,V_2\ldots, V_k$, in which  $H[V_i]$ satisfies the property $P$, for each $1\leq i\leq k$. This extension of graph coloring was stated by Harary in 1985~\cite{MR778402}. A  particular state, when $P$ is the property of being acyclic, $\chi(H,P)$ is said  the vertex arboricity of $H$. The vertex arboricity   of $H$ is shown  $\alpha(H)$ and is defined as the last number of subsets in a partition of the vertex set of $H$, so that any subset induces an acyclic subgraph. When $P$ is the property that  each color class contains no copy of $G$, we write $\chi_{G}(H)$ instead of  $\chi(G,P)$, which is called the $G$-free chromatic number. Due to this, we say a graph $H$ has a $k$-$G$-free coloring if there exists a map $g : V(H) \longrightarrow \{1,\ldots,k\}$, so that each of the  color classes of $g$ be $G$-free.

We use $\alpha_L(H)$ to denote the list vertex arboricity of $H$, which is the least integer $k$, such that there exists an acyclic $L$-coloring for each list assignment $L$ of $H$, in which $k\leq |L(v)|$. So, $\alpha(H)\leq \alpha_L(H) $ for any graph $H$. It has been proved that for each  graph $H$, $\alpha(H)\leq \lceil \frac{\Delta(H)+1}{2}\rceil$~\cite{Chartrand}, and $\alpha_L(H)\leq \lceil \frac{\Delta(H)+1}{2}\rceil$~\cite{ xue2012list}. When $H$ is not a complete  graph or a cycle of odd order, then  $\alpha(H)\leq \lceil \frac{\Delta}{2}\rceil$\cite{Kronk}. Also, it has  been shown that  $\alpha(H)\leq 3, \alpha_L(H)\leq 3$ if $H$ be a planar graph,  \cite{Chartrand,Hedet, xue2012list}. As a generalized result of Theorem \ref{thm1} and \ref{thm2}, Lingyan Zhe, and  Baoyindureng Wu have proven the following theorem \cite{zhen2009list}.
\begin{alphtheorem}{\rm\cite{zhen2009list}}\label{m2}
 If $|V(H)|\leq 2\alpha(H)+\sqrt{2 \alpha(H)}-1$, then:
 \[\alpha_L(H)=\alpha(H).\]
\end{alphtheorem}
Assume that each vertex $v\in V(H)$ is assigned a set $L(V)$ of colors, told a color list. Set $c(L) = \{c(v): v\in V(H)\}$. An $L$-coloring $c$ is called  $G$-free,  so that:
\begin{itemize}
	\item  	$c(v)\in L(v)$ for each $v\in V(H)$.
	\item  $ H[V_i]$ is $G$-free for each $i=1,2,\ldots, L$.
	
\end{itemize}
If there exists an $L$-coloring of $H$, then $H$ is said to be $L$-$G$-free-colorable. A graph $H$ is said $k$-$G$-free-choosable if there exists an $L$-coloring for any list-assignment $L$ satisfying  $|L(V)|\geq k$ for each $v\in V(H)$, and  $ H[V_i]$ be $G$-free for each $i=1,2,\ldots, L$. For given two graphs $H$ and $G$, the $\chi_G^L(H)$ of $H$ is the minimum integer $k$, if $H$ be $k$-$G$-free-choosable.

In this article, we shall use Ohba’s idea to show  similar results in terms of $G$-free coloring,
 and list $G$-free coloring  of graphs. In particular, in this article, we prove the subsequent results.

\begin{theorem}\label{mth1}  
	Suppose that $H$, $H'$, and $G$  are three graphs, where $\delta(G)=\delta$,  $H$ is a $k$-$G$-free choosable, and $H'$ is a $k'$-$G$-free choosable. Suppose that $S$ and $S'$ be the maximum subsets of $V(H)$ and $V(H')$, respectively, so that $H[S]$ is $G$-free and  $H'[S']$ is $G$-free. In this case, if either $(|S'|-1)(|V(H)|+|S'|)\leq |S'|\delta (k+1)$ or $(|S|-1)(|V(H')|+|S|)\leq |S|\delta (k'+1)$, then $H\oplus H'$ is a $(k+k')$-$G$-free-choosable, that is: \[\chi_G^L(H\oplus H')\leq \chi_G^L(H)+\chi_G^L(H')=k+k'.\]	
\end{theorem}
 \begin{theorem}\label{mth2}
	If $|V(H)|\leq \delta \chi_G(H)+\sqrt{\delta \chi_G(H)}-(\delta-1)$, then:
	\[\chi_G^L(H)=\chi_G(H).\]  
\end{theorem}	
\begin{theorem}\label{mth3}  
For each two connected graphs  $H$ and $G$, we have:  
\[\chi^L_{_G}(H)\leq\lceil \frac{\Delta(H)}{\delta(G)}\rceil+1.\]	
\end{theorem}
\begin{theorem}\label{mth4}
	Assume that  $\R$ be a collection of all $d$-regular graphs. For each  arbitrary graph $H$, there exists a non-negative integer $n'$, so that $\chi_{\R}(H\oplus K_n)=\chi^L_{\R}(H\oplus K_n)$, where $n$ is integer so that  for which $n'\leq n$.	 
\end{theorem}
\section{Proof of the Main results}	
 Before proving the main theorems, we need  some basic results. Suppose that $H$ and $G$  are two graphs, and let $L$ be a list-assignment color to $V(H)$. Assume that $S= {v_1, v_2,\ldots, v_m}\subseteq V(H)$. Set $L(S)= \cup_{i=1}^{i=m}L(v_i)$. Now, we have the following lemma.

\begin{lemma}\label{l1}  
Suppose that $H$ and $G$  be two graphs, where $\delta(G)=\delta$. Let $H$ is not $L$-$G$-free colorable. Hence there is a subset of $ V(H)$ say $S$, such that:
\[|S| > \delta|L(S)|.\]
\end{lemma}
\begin{proof}
By contradiction, suppose that $|S| \leq \delta|L(S)|$  for each subset $S$ of $V(H)$. Now, define the bipartite graph $B$ with bipartition $(Y, Y')$, where $Y=V(H)$, and $Y'$ is the $\delta$ copies	 of $L(V(H))$. Also, for any member of $Y$ say $y$, and each member of $Y'$ say $y'$, $yy'\in E(B)$ if $y'\in L(y)$. It is clear to say that $N_B(W)=\delta|L(W)|$, for each $W\subseteq V(H)$. Now, let $S\subseteq V(H)$. By the assumption,  $N_B(S)=\delta|L(S)|\geq |S|$, therefore,   by Hall’s theorem, there exists a matching $M$ that saturates $V(H)$. Consider $y$ and  color $y$ with the  color matched by it in $M$. As $Y'$ is a $\delta$ copy of $L(V(H))$. Any member $y'$ of $L(V(H))$ appears in at most $\delta$ times as an end vertex of some edges of the  $M$. Which means that, any color of $L(H)$ is assigned in at most $\delta$ vertices of $V(H)$. Hence, we achieve a $L$-$G$-free coloring of $V(H)$, which is impossible. Therefore the proof is complete.
\end{proof}
To prove the following lemma, we  use Ohba’s notion.
\begin{lemma}\label{l2}  
	Suppose that $H$, $H'$, and $G$  are three graphs, where $\delta(G)=\delta$, $|V(H')|=n$, and $H'$ is $G$-free, i.e $G\nsubseteq H'$. If  $H$ be $k$-$G$-free choosable, and $(n-1)(|V(H)|+n)\leq n\delta (k+1)$, then $H\oplus H'$ is  $(k+1)$-$G$-free choosable, that is:
	\[\chi_G^L(H\oplus H')\leq \chi_G^L(H)+1=k+1.\]
\end{lemma}
\begin{proof}
The proof proceeds by induction on $|V(H')|$. Suppose that $L$ is a $(k + 1)$-list assignment of   $H\oplus H'$. For the induction basis, first assume that $n=1$, and $V(H')=\{v'\}$. We color $v'$ by one of the members of $L(v')$, say $c'$. For each member of $V(H)$ say $v$, assume that $L(v)\setminus \{c'\}=L'(v)$. As for any member of $ V(H)$ say $v$, $|L'(v)|\geq |L(v)|-1= k$, and $H$ is $k$-$G$-free choosable. Thus $H\oplus H'$ is  $(k+1)$-$G$-free choosable. So, assume that $n\geq 2$. Suppose that $V'=V(H')=\{v'_1,v'_2,\ldots,v'_n\}$. Now, by considering $\bigcap\limits_{v'\in V'} L(v') $, we have two cases as follow:\\

{\bf Case 1:} $\bigcap\limits_{v'\in V'} L(v')\neq \emptyset$.

Set $c'\in \bigcap\limits_{v'\in V'} L(v')$ and assign $c'$ to each member of $V(H')$. Therefore, as $G\nsubseteq H'$, one can check that $H\oplus H'[V_{c'}]$ is $G$-free. Now,  for each member of $V(H)$ say $v$, assume that $L(v)\setminus \{c'\}=L'(v)$. As for each member of $ V(H)$ say $v$, $|L'(v)|\geq |L(v)|-1= k$, and $H$ is $k$-$G$-free choosable, so $H\oplus H'$ is a $(k+1)$-$G$-free choosable.\\

{\bf Case 2:} $\bigcap\limits_{v'\in V'} L(v')= \emptyset$.

By contradiction, assume  that $H\oplus H'$ has no $L$-$G$-free coloring. By Lemma \ref{l1}, there must exist $S\subseteq V( H\oplus H')$ with $|S|>\delta |L(S)|$.  Suppose that $S$ be the maximal subset of $V(H\oplus H')$ so that $|S|>\delta |L(S)|$. 
 
 Suppose that $V'\subseteq S$. As $\bigcap\limits_{v'\in V'} L(v')= \emptyset$, any color in $L(V')$ appears in the lists of at most $n-1$ vertices of $V'$, therefore one can say that:
\begin{equation}\label{e1}
 |L(S)|\geq |L(V')|\geq \frac{n}{n-1}(k+1). 
\end{equation}
On the other hand, by the assumption, as $(n-1)(|V(H)|+n)\leq n\delta (k+1)$, it can be check that:
\begin{equation}\label{e2}	
	|S|\leq |V(H\oplus H')|\leq \delta\frac{n}{n-1}(k+1).
\end{equation}
Therefore, by Equations \ref{e1}, \ref{e2}, we have $|S|\leq \delta |L(S)|$, a contradiction. So, we may suppose that $V'\nsubseteq S$, that is  $|V'\setminus S|\neq 0$. 

Set $V''=V(H\oplus H')\setminus S$. Assume that $L(v'')\setminus L(S)=L'(v'')$ for each $v''\in V''$. By considering $H[S]$ and $H[V'']$, we have two claims as follow:
\begin{claim}\label{c1}
	$H[S]$ has a $L$-$G$-free coloring.
\end{claim} 
\begin{proof}[Proof of the Claim\ref{c1}]
As $(n-1)(|V(H)|+n)\leq n\delta (k+1)$, so $(n-2)(|V(H)|+n-1)\leq (n-1)\delta (k+1)$, because this operation is an increasing operation. Therefore, by the induction hypothesis,  $H\oplus H'\setminus\{v'\}$ is $(k+1)$-$G$-free list colorable, for any $v'\in V(H')$. Since $H[S] \subseteq  H\oplus H'\setminus\{v'\}$ it can be said  $H[S]$ has an  $L$-$G$-free coloring.
\end{proof}
\begin{claim}\label{c2}
	$H[V'']$ has an $L'$-$G$-free coloring.
\end{claim} 
\begin{proof}[Proof of the Claim\ref{c2}]
 By  contradiction, let 	$H[V'']$ is not $L'$-$G$-free colorable. Therefore, by Lemma \ref{l1}, there exists a set $S'$ of $V''$, so that $|S'|> \delta |L'(S')|$. In the other hand, $|L(S\cup S')|= |L(S)|+|L'(S')|< \frac{1}{\delta}|S+S'|$, therefore,  $\delta|L(S\cup S')|<|S+S'|$ which contradicts to the maximality of $S$. Hence, $|S'|\leq \delta |L'(S')|$ for each $S'\subseteq V''$, which means that 	$H[V'']$ has an $L'$-$G$-free coloring.
\end{proof}
 
Therefore, by Claim \ref{c1}, and Claim \ref{c2}, we can say that  $H\oplus H'$  has a $L$-$G$-free colorable, a contradiction. Hence:
	\[\chi_G^L(H\oplus H')\leq \chi_G^L(H)+1.\]
Which means that the proof is complete.
\end{proof}	
In the following, by using Lemma \ref{l2}, we prove the first main result, namely  Theorem \ref{mth1}.
\begin{proof}[\bf Proof of the Theorem \ref{mth1}]
	
Without loss of generality (W.l.g), suppose that:
 \[ (|S'|-1)(|V(H)|+|S'|)\leq |S'|\delta (k+1).\] 
 Where $S'$ is the maximum subset of $V(H')$  so that $G\nsubseteq H'[S']$. Therefore,  as $H'$ is  $k'$-$G$-free choosable and $\chi_G(H')\leq \chi_G^L(H')$, so $\chi_G(H')\leq k'$. For $i=1,2,\ldots, k'$, set $V'_i\subseteq V(H')$ where $V'_1$ is the maximum subset of $V(H')$  so that $G\nsubseteq H'[V_1]$ and for each $i\geq 2$, $V'_i$ be the  maximum subset of $V(H')\setminus (\cup_{j=1}^{j=i-1} V'_j)$, and $(H'\setminus (\cup_{j=1}^{j=i-1} V'_j))[V'_i]$ be $G$-free. It can be said that $|V'_1|= |S'|$, and $|V'_{i}|\leq |V'_{i-1}|$ for each $i\geq 2$. So, as $(|S'|-1)(|V(H)|+|S'|)\leq |S'|\delta (k+1)$, $S'$ is the maximum subset of $V(H')$, $G\nsubseteq H'[S']$ and $|V_1|= |S'|$, then by Lemma \ref{l2} we have the following:
\begin{equation}\label{e3}	
	\chi_G^L(H\oplus H'[V'_1])\leq \chi_G^L(H)+1. 
\end{equation}
For each $i\geq 1$, set $H_i=H_{i-1}\oplus H'[V'_i]$, where $H_0=H$. Therefore, it can be say that  the following claim is true:
\begin{claim}\label{c3}
  For each $i\geq 1$, $H\oplus H'[\cup_{j=1}^{j=i}V'_j]\subseteq H_i$.
\end{claim}
\begin{proof}[Proof of Claim \ref{c3}]
	It is clear that $V(H\oplus H'[\cup_{j=1}^{j=i}V'_j])=V(H_i)$, also one can say that $E(H\oplus H'[\cup_{j=1}^{j=i}V'_j])=E(H_i)\setminus E'$, where $E'=\{v_{j}v_{j'}~: j\leq j', ~~v_j\in V'_j, ~v_{j'}\in V'_{j'}~~for~ each~~j\le j'\leq i\}$, which means that the  claim is true.
\end{proof}
Also, since $(|S'|-1)(|V(H)|+|S'|)\leq |S'|\delta (k+1)$, $|V'_1|=|S'|$, and  $|V'_{i}|\leq |V'_{i-1}|$ for each $i\geq 2$, it can be said that the following claim is true:
\begin{claim}\label{c4} For each $i\in \{1,2,\ldots,k'\}$, we  have:	
 \[ (|V_i'|-1)(|V(H)|+|V_i'|)\leq |V_i'|\delta (k+1).\] 
\end{claim} 
\begin{proof}[Proof of Claim \ref{c4}]	
 For $i=1$, since $|V'_1|=|S'|$, and  $(|S'|-1)(|V(H)|+|S'|)\leq |S'|\delta (k+1)$, the proof is complete. So, suppose that $2\leq j\leq k'$, also for each $2\leq j\leq k'$ suppose that $|V'_1|-|V'_j|= t_j$.  Therefore, we need to show that  $(|V'_{j}|-1)(|V(H)|+|V'_{j}|)\leq |V'_{j}|\delta (k+1)$. As $|V'_1|-|V'_j|= t_j$, so:
 
\[ (|V'_{j}|-1)(|V(H)|+|V'_{j}|)=(|V'_1|-1-t_j)(|V(H)|+|V'_1|-t_j)\] 
	
\[=(|V'_1|-1)(|V(H)|+|V'_1|)-(t_j(|V'_1|-1)-t_j(|V(H)|+|V'_1|-t_j) \] 

\[\leq (|V'_1|)\delta(k+1)-(t_j(|V'_1|-1)-t_j(|V(H)|+|V'_1|-t_j)\]

\[= (|V'_j|)\delta(k+1)+t_j\delta(k+1)-(t_j(|V'_1|-1)-t_j(|V(H)|+|V'_1|-t_j)\]

\[= (|V'_j|)\delta(k+1)+t_j(\delta(k+1)-|V(H)|-2|V'_1|+1+t_j)\]  

\[= (|V'_j|)\delta(k+1)+t_j(k\delta+\delta-|V(H)|-|V'_1|-|V'_j|+1)\] 
  
So for each $j\in \{2,3,\ldots, k'\}$ we have the following equation:
\begin{equation}\label{e}
(|V'_{j}|-1)(|V(H)|+|V'_{j}|)\leq  (|V'_j|)\delta(k+1)+t_j(k\delta+\delta-|V(H)|-|V'_1|-|V'_j|+1)
\end{equation}
  
Now by Equation \ref{e}, it is sufficient  to show  that   for each $j\in \{2,3,\ldots, k'\}$ we have the followings:
\[ t_j(k\delta+\delta-|V(H)|-|V'_1|-|V'_j|+1)\leq 0.\]  
It is easy to say that $|V(H)|\geq (k-1)\delta+1$, also one can say that $|V'_1|+|V'_j|\geq 2\delta$, which means that:
  \[t_j(k\delta+\delta-|V(H)|-|V'_1|-|V'_j|+1)\leq 0.\] 
Hence, for each $i\in \{1,2,\ldots,k'\}$:
\begin{equation}\label{e4}
  (|V_i'|-1)(|V(H)|+|V_i'|)\leq |V_i'|\delta (k+1). 
\end{equation}
Therefore, Equation \ref{e4} shows that the proof of the claim is complete.
\end{proof}	
Hence, by Claim \ref{c4} and by Lemma \ref{l2}, for each $i\in \{1,,\ldots,k'\}$ we can say that:
\begin{equation}\label{e5}	
	\chi_G^L(H\oplus H'[V'_i])\leq \chi_G^L(H)+1\leq k+1. 
\end{equation}
So, regarding Equation \ref{e5}, and $\chi_G(H'[\cup_{j=1}^{j=i} V'_{j}])=i$, by assigning $i$ new separate and unique colors to each $V'_i$, it can concluded that $\chi_G^L(H\oplus H'[\cup_{j=1}^{j=i} V'_{j}])\leq k+i$, for each $i\in \{1,2,\ldots,k'\}$. Also,  since $H\oplus H'\cong H\oplus H'[\cup_{i=1}^{i=k'}V'_i]$, so:
\begin{equation}\label{e6}	
\chi_G^L(H\oplus H')= \chi_G^L(H\oplus H'[\cup_{i=1}^{i=k'}V'_i])\leq k+k'=\chi_G^L(H)+\chi_G^L(H'). 
\end{equation}
Equation \ref{e6}  means that the proof is complete.
\end{proof}	
Before proving Theorem \ref{mth2}, we need two lemmas, which we state and prove in the following. In the following lemma, we determine an upper bond for $\chi_G^L(H)$, for each graph $H$ and $G$.
\begin{lemma}\label{l3}  
	Suppose that $H$ and $G$  be two graphs, where $\delta(G)=\delta$, $|V(H)|=n$, then:
	\[\chi_G^L(H)\leq \lceil \frac{n}{\delta}\rceil\]
\end{lemma}
\begin{proof}
The proof proceeds by induction on $|V (H)|$. It is clear that $\chi_G^L(H)\leq \lceil \frac{n}{\delta}\rceil$  when $|V(H)|\in \{1,2,\ldots,\delta\}$. Hence,  assume that $|V(H)| \geq \delta+1$ and suppose that $L$ is  a $\lceil \frac{n}{\delta}\rceil$-list apportion of $H$. If there exist $v_1,v_2,\ldots, v_{\delta-1}$, and $v_{\delta}$ in $H$ so that $L(v_1) \cap L(v_2)\cap \ldots \cap L(v_{\delta})\neq \emptyset$, then  we catch a color $c \in L(v_1) \cap L(v_2)\cap \ldots \cap L(v_{\delta})$ and allocate $c$ to $v_1,v_2,\ldots, v_{\delta-1}$, and $v_{\delta}$  and set $L(v)\setminus \{c\}=L(v)$ for any $v \in V(H)\setminus \{v_1,v_2,\ldots, v_{\delta}\}$. Regarding $|L(v)|\geq|L(v)|-1\geq \lceil \frac{n-\delta}{\delta}\rceil$, and by the induction hypothesis,  $H\setminus \{v_1,v_2,\ldots, v_{\delta}\}$ has an $L$-$G$-free coloring. Therefore, $H$ is $L$-$G$-free colorable. Otherwise, if for any $\delta$ vertices $v_1,v_2,\ldots, v_{\delta-1}$, and $v_{\delta}$ of $H$, $L(v_1) \cap L(v_2)\cap \ldots \cap L(v_{\delta})= \emptyset$, then  it is easy to say that $H$ is $L$-$G$-free colorable, by considering $\lceil \frac{n}{\delta}\rceil$ class of size at most $\delta$.	
\end{proof}	
Suppose that $H$ and $G$  are two graphs, where $\delta(G)=\delta$, $|V(H)|=n$,  and  we may suppose that $g:V(H)\rightarrow \{1,2,\ldots, \chi_G(H)\}$ is a $\chi_G(H)$-coloring of $H$, so that for each $i\in [\chi_G(H)]$, $G\nsubseteq H[V_i]$  and $V_i=\{v\in V(H):~~g(v)=i\}$. For each $i$, suppose that $|V_i|=n_i$, and w.l.g assume that $n_1\geq n_2\geq\ldots \geq n_{\chi_G(H)}$. It is easy to say that $n_i\geq \delta$ for each $i\in [\chi_G(H)-1]$. Now by considering $n_i$, we have the following lemma. 
\begin{lemma}\label{l4}  If $(n_1-1)|V(H)|\leq n_1\delta \chi_G(H)$. Then:
	\[\chi_G^L(H)=\chi_G(H).\]
\end{lemma}
\begin{proof}
Since $\chi_G(H)\leq \chi_G^L(H)$, it do to prove that $H$ is $\chi_G(H)$-$G$-free choosable(list colorable). The proof proceeds by induction on $\chi_G(H)$. If $H$ be $G$-free, then  $\chi_G^L(H)=\chi_G(H)=1$, so suppose that $H$ has at least one copy of $G$ as a subgraph, that is $2\leq \chi_G(H)$. Since $n_1\geq \delta$, then by lemma assumption one can say that:
\begin{equation}\label{e7}	
	 |V(H)|\leq \frac{n_1\delta}{n_1-1}\chi_G(H).
\end{equation}
If $n_1=\delta$, then as $n_1$ is maximal, so $\chi_G(H)=\lceil \frac{|V(H)|}{\delta}\rceil$, hence by Lemma \ref{l3}, $\chi_G^L(H)=\chi_G(H)$. Now, suppose that $n_1\geq \delta +1$. Hence:
\[|V(H)\setminus V_1|=|V(H)|-n_1\]
\[\leq \frac{n_1\delta}{n_1-1}\chi_G(H)-n_1\]
\[= \frac{n_1}{n_1-1}(\delta\chi_G(H)-n_1+1)\]
\[\leq \frac{n_1}{n_1-1}(\delta\chi_G(H)-\delta)\]  
\[\leq \frac{n_2\delta}{n_2-1}(\chi_G(H)-1)\] 
So, this inequality  and induction hypothesis  implies that $H\setminus V_1$  is $(\chi_G(H)-1)$-$G$-free choosable, because $V_2$ is the maximal color class of  $H\setminus V_1$. Now, regarding Equation \ref{e7} and the fact that  $H\setminus V_1$ is $(\chi_G(H)-1)$-$G$-free list colorable, Lemma \ref{l2} implies that $H$ is $\chi_G(H)$-$G$-free choosable, which means that the proof is complete.
\end{proof}  
In the following, we prove Theorem \ref{mth2} and  Theorem \ref{mth3} by using Lemma \ref{l4}.
\begin{theorem}\label{t3}{\bf (Theorem\ref{mth2})} Suppose that $H$ and $G$  are two graphs, where $\delta(G)=\delta$. If  we have, $|V(H)|\leq \delta \chi_G(H)+\sqrt{\delta \chi_G(H)}-(\delta-1)$, then:
	\[\chi_G^L(H)=\chi_G(H).\]
\end{theorem}
\begin{proof}
Assume that $g:V(H)\rightarrow \{1,2,\ldots, \chi_G(H)\}$ is the $\chi_G(H)$-coloring of $H$, so that for each $i\in [\chi_G(H)]$, $G\nsubseteq H[V_i]$ and $V_i=\{v\in V(H):~~g(v)=i\}$. For each $i$, suppose that $|V_i|=n_i$, and w.l.g assume that $n_1\geq n_2\geq\ldots \geq n_{\chi_G(H)}$. It is easy to say that $n_i\geq \delta$ for each $i\in [\chi_G(H)-1]$, and $n_{\chi_G(H)}\geq 1$. It can be checked that:
\begin{equation}\label{e88}
  n_1\leq |V(H)|-(\delta\chi_G(H)-(2\delta-1)).
\end{equation}
  Otherwise one can check that $\sum_{i=1}^{i=\chi_G(H)} n_i\geq |V(H)|+1$, a contradiction.  Therefore, by assumption lemma  and by Equation \ref{e88} we have the next equation:
\begin{equation}\label{e8}	
 n_1\leq \delta \chi_G(H)+\sqrt{\delta \chi_G(H)}-(\delta-1)-(\delta\chi_G(H)-(2\delta-1))=\sqrt{\delta \chi_G(H)}+\delta.
\end{equation}
Hence, by Equation \ref{e8} we have the following:
\[\frac{n_1}{n_1-1}\delta \chi_G(H)\geq \frac{\sqrt{\delta \chi_G(H)}+\delta}{\sqrt{\delta \chi_G(H)}+(\delta-1)}\times \delta \chi_G(H)\]
\[=\delta \chi_G(H)+\frac{\delta \chi_G(H)}{ \sqrt{\delta \chi_G(H)}+(\delta-1)}\]
\[>\delta \chi_G(H)+\frac{\delta \chi_G(H)-(\delta-1)}{ \sqrt{\delta \chi_G(H)}+(\delta-1)}\]
\[=\delta \chi_G(H)+ \sqrt{\delta \chi_G(H)}-(\delta-1)\geq |V(H)|\]
So, by Lemma \ref{l4}, $\chi_G^L(H)=\chi_G(H)$, which means that the proof is complete.
\end{proof}  
Now by Theorem \ref{t3}, one can say that the following results are valid. For  this, first we need the following definition.
\begin{definition}{\bf$\G$-free graph:}
	Suppose that  $\G$ is a collection of  some graphs, a given graph $H$ is called  $\G$-free, if $H$ contains no copy of $G$ as a subgraph for each $G\in \G$. When $\G=\{G\}$, then $H$ is $G$-free if $G\nsubseteq H$. In this attention, we say a graph $H$ has a $k$-$\G$-free coloring if there exists a map $g : V(H) \longrightarrow \{1,\ldots,k\}$ so that each of the  color classes of $g$ is $\G$-free i.e. contain no copy of each member of $\G$ as a subgraph.  For given  graphs $H$ and $\G$, the $\chi_{\G}(H)$ of $H$ is the minimum integer $k$ so that $H$ is $k$-$\G$-free coloring.
\end{definition}
\begin{corollary}\label{cl1}Suppose that $\G$ be a collection of all $d$-regular graphs, where for each $G\in \G$ we have:
\[|V(H)|\leq d\chi_G(H)+\sqrt{d \chi_G(H)}-(d-1).\] 
Then:
\[\chi_{\G}^L(H)=\chi_{\G}(H).\]  	 
\end{corollary}
\begin{proof}
	Consider $G\in \G$, so that for which  $\chi_{G'}(H)\leq \chi_{G}(H)$ for each $G'\in \G$. Therefore, one can say that $\chi_{\G}(H)\leq \chi_{G}(H)$, also by assumption and by Theorem \ref{t3} we have  $\chi_{G}(H)=\chi^L_{G}(H)$, that is $\chi_{\G}(H)\leq \chi^L_{G}(H)$. So as $\chi^L_{G}(H)\leq \chi^L_{\G}(H)$ we have $\chi_{\G}(H)\leq \chi^L_{\G}(H)$. Now we have to show that $\chi^L_{\G}(H)\leq \chi_{\G}(H)$. 	Consider $G''\in \G$, so that for which  $\chi^L_{\G}(H)\leq \chi^L_{G''}(H)$. Therefore,  by assumption and by Theorem \ref{t3}, we have  $\chi_{G''}(H)=\chi^L_{G''}(H)$, and as 
 $\chi_{G''}(H)\leq \chi_{\G}(H)$, it can be said that $\chi^L_{\G}(H)\leq \chi_{\G}(H)$, which means that the proof is complete.
\end{proof}
	
In corollary  \ref{cl1}, if we take   $d=2$ hence $\G=\{ C_n, n \geq 3\}$ and  for any
arbitrary graph for $H$, then it is easy to say that  we get Theorem \ref{m2}. Also, for $d=1$, we have $\G=\{K_2\}$. Hence, as  $|V(H)|\leq  \chi(H)+\sqrt{ \chi(H)}\leq  \chi(H)+\sqrt{2 \chi(H)}$, therefore by corollary  \ref{cl1}, if we take   $d=1$, then  we get Theorem \ref{thm1}.
\begin{theorem}{\bf(Theorem \ref{mth3})}
For each two connected graphs  $H$ and $G$, we have:  	  
	\[\chi^L_{_G}(H)\leq\lceil \frac{\Delta(H)}{\delta(G)}\rceil+1.\]
\end{theorem}
\begin{proof} 
Suppose that $H$ and $G$ are two connected graphs that satisfy the theorem conditions. Assume that $L$ is a $k$-list assignment of $H$, where
$k=\lceil \frac{\Delta(H)}{\delta(G)}\rceil+1$. Suppose that $H$ is not $G$-free. Otherwise, it is easy to say that  $\chi^L_{_G}(H)=1\leq \lceil \frac{\Delta(H)}{\delta(G)}\rceil+1$. Hence  assume that $G\subseteq H$, that is $2\leq \chi_{_G}(H)$. Suppose that $v_1, v_2,\ldots, v_n$ be an ordering of $V(H)$ so that for each $i$, vertex $v_i$, have at most $\Delta(H)$ low-indexed neighbors in $H$. Now we color the vertices of $H$ from their lists by this order. Assume that $v_1, v_2,\ldots, v_i$ has been colored and every subset of $v_1, v_2,\ldots, v_i$ which colored with the same color be  $G$-free. Since $v_{i+1}$ has at most $\Delta(H)+1-\delta(G)$ neighbors in $v_1, v_2,\ldots, v_i$, there are at most  $\lceil \frac{\Delta(H)}{\delta(G)}\rceil$ of this $G$ containing at least $\delta$ neighbors of $v_{i+1}$. Therefore, it can be said that for $v_{i+1}$, there exists at least one color available in $L(v_{i+1})$ which appears at most once among the  low-indexed  neighbor of $v_{i+1}$. We assign such  color to  $v_{i+1}$. Therefore, we gain a $G$-free coloring of $H[{v_1, v_2, \ldots, v_{i+1}}]$. By continuing this method until $i+1=n$, we obtain a $G$-free $L$-coloring of $H$. 
\end{proof}  
\subsection{Proof of Theorem \ref{mth4}} 
 Before proving  Theorem \ref{mth4}, we need to some results. Suppose that   $\R$ is a collection of all $d$-regular graphs and let $H$ be an  arbitrary graph.  Hence, we have the following lemma:
\begin{lemma}\label{le5}
 Suppose that $n$ be a positive integer, and $H$ be an arbitrary graph. So:
 \[\chi_{\R}(H\oplus K_n)\geq \frac{\chi_{\R}(H)+n}{d}.\]
\end{lemma}
\begin{proof} 
 Suppose that $t=\chi_{\R}(H\oplus K_n)$ and let $g$ be a $t$-$\G$-free coloring of $H\oplus K_n$. For each $i\in \{1,2,\ldots, \chi_{\R}(H)\}$, let:
\[V_i=\{v\in V(H\oplus K_n)~:~g(v)=i\}.\]
Since $H[V_i]$ is a $\R$-free, and $\R$ is a collection of all $d$-regular graphs, so $K_{d+1}\nsubseteq H[V_i]$ for each $i$. Therefore, each $V_i$ contains at most $d$ vertices of $V(K_n)$. Now, if $t_j=|\{i~:~ 1\leq i\leq t~~and~~ |V_i\cap V(K_n)|=j\}|$, and $k_j\neq 0$, then $j\leq d$.  Therefore, $t=t_0+t_1+\ldots+t_d$, and it is easy to say that $n=\sum_{m=1}^{m=d}mt_m$, also one can say that $\chi_{\R}(H)\leq  \sum_{m=0}^{m=d-1}t_m$. Hence,
 \[ dt=d(\sum_{m=0}^{m=d}t_m)=dt_0+\sum_{m=1}^{m=d-1}(d-m)t_m+\sum_{m=1}^{m=d}mt_m\]
 \[\geq (d-1)t_0+\sum_{m=0}^{m=d-1}t_m+\sum_{m=1}^{m=d}mt_m \]
 \[\geq \chi_{\R}(H)+n\]
 Therefore, $t=\chi_{\R}(H\oplus K_n)\geq \frac{\chi_{\R}(H)+n}{d}$, which means that  proof is complete.
\end{proof} 
\begin{theorem}\label{t0}
 Assume that  $\R$ is a collection of all $d$-regular graphs where $d
 \geq 1$. For each  arbitrary graph $H$, there exists a non-negative integer $n'$, so that $\chi_{\R}(H\oplus K_n)=\chi^L_{\R}(H\oplus K_n)$, for any integer $n$ with $n\geq n'$.	 
\end{theorem}
\begin{proof}
Suppose that $m= |V (H)|$, and $g: V (H) \rightarrow \{1,2,\ldots, \chi_{\R}(H\oplus K_n)\}$  be an optimal $\R$-free coloring of $H\oplus K_n$, with a  color class of sizes $ n_1\geq n_2\geq \ldots\geq n_{\chi_{\R}(H\oplus K_n)}$. One can check that $d \leq n_1\leq m+d$ for sufficiently large $n$. By Lemma \ref{le5}, $\chi_{\R}(H\oplus K_n)\geq \frac{\chi_{\R}(H)+n}{d}$. Observe that $m+n\leq \frac{d(m+d)}{m+d-1} \frac{dn_1}{n_1+1-d}\leq \frac{dn_1}{n_1+1-d}\chi_{\R}(H\oplus K_n)$ for sufficiently large $n$. Therefore,
\[(n_1+1-d)(m+n) \leq dn_1\chi_{\R}(H\oplus K_n)\]
 Hence, by Lemma \ref{l4}, $\chi_{\R}(H\oplus K_n)=\chi^L_{\R}(H\oplus K_n)$, which means that the proof is complete.
\end{proof}

  The following theorem has been proved by  Ohba \cite{ohba2002chromatic}: 
\begin{theorem}\label{t1}{\rm\cite{ohba2002chromatic}}
Let $H$ be a graph. Hence there is a positive integer $n'$ so that $\chi(H\oplus K_n)=\chi_L(H\oplus K_n)$, for each integer $n$ with $n'\leq n$.
\end{theorem}
 As a generalized result of Theorem \ref{t1}, Lingyan Zhe, and  Baoyindureng Wu have proven the following theorem \cite{zhen2009list}.
\begin{theorem}\label{t2}{\rm\cite{zhen2009list}}
Let $H$ be a graph. Hence there is a positive integer $n'$ so that $\alpha(H\oplus K_n)=\alpha_L(H\oplus K_n)$, for each integer $n$ with $n'\leq n$.
\end{theorem} 
 In Theorem \ref{t0}, if we take $d=1$ hence $\R=\{K_2\}$, then we get Theorem \ref{t1}. Also, by setting $d=2$, we have $\R=\{ C_n, n \geq 3\}$ and  for any
 arbitrary graph for $H$, then it is easy to say that  we get Theorem \ref{t2}.
\subsection{ Some research problems related to the contents of this paper.} 
In this section, we propose some research problems related to the contents of this paper. The following conjecture has been proposed by  Ohba \cite{ohba2002chromatic}: 
\begin{conjecture}\label{co1}{\rm\cite{ohba2002chromatic}}
	If $|V(H)|\leq  2\chi(H)+1$, then:
	\[\chi_L(H)=\chi(H).\]   
\end{conjecture}
Jonathan A. Noel et.al in \cite{noel2015proof} have proven the conjecture\ref{co1}. The first problem concerns  Theorem \ref{mth2} and Conjecture \ref{co1}, as we address below:  
\begin{problem}
	For each two connected graphs  $H$ and $G$, find a constant $M$, so that if  $|V(H)|\leq M\chi_G(H)$, Then: 
	\[\chi^L_{G}(H)=\chi_{G}(H).\]
\end{problem}

Y.Rowshan  and A.Taherkhani in \cite{rowshan2020catlin} have proven the following theorem.
\begin{alphtheorem}\label{2th}
Suppose that $H$ and $G$ be two connected graphs while $H$ satisfies the followinf items:
	\begin{itemize}
		\item If $G$ is regular, then $H\ncong G$.
	    \item If $G\cong K_{\delta(G)+1}$, then $H$ is not $K_{k\delta(G)+1}$.
	    \item If $G\cong K_2$, then $H$ is neither a complete graph nor an  odd cycle.
	\end{itemize}
	Then:
	\[\chi_{_G}(H)\leq\lceil \frac{\Delta(H)}{\delta(G)}\rceil.\]
	And    one of those color classes is a maximum induced  $G$-free subgraph in $H$.
\end{alphtheorem}

 The second problem concerns  Theorem \ref{mth3} and \ref{2th}, as we address below: 
\begin{problem}
Suppose that $H$ and $G$ be two connected graphs while $H$ satisfies the following items:
	\begin{itemize}
		\item If $G$ is regular, then $H\ncong G$.
		\item If $G\cong K_{\delta(G)+1}$, then $H$ is not $K_{k\delta(G)+1}$.
		\item If $G\cong K_2$, then $H$ is neither a complete graph nor an  odd cycle.
	\end{itemize}
	Then:  
	\[\chi^L_{_G}(H)\leq\lceil \frac{\Delta(H)}{\delta(G)}\rceil.\]
	And    one of those color classes is a maximum induced  $G$-free subgraph in $H$.
\end{problem}
 
\bibliographystyle{plain}
\bibliography{l.c}

\begin{thebibliography}{10}

\bibitem{Chartrand}
Gary Chartrand, Hudson~V. Kronk, and Curtiss~E. Wall.
\newblock The point-arboricity of a graph.
\newblock {\em Israel J. Math.}, 6:169--175, 1968.

\bibitem{MR778402}
Frank Harary.
\newblock Conditional colorability in graphs.
\newblock In {\em Graphs and applications ({B}oulder, {C}olo., 1982)},
  Wiley-Intersci. Publ., pages 127--136. Wiley, New York, 1985.

\bibitem{Hedet}
Stephen Hedetniemi.
\newblock On partitioning planar graphs.
\newblock {\em Canad. Math. Bull.}, 11:203--211, 1968.

\bibitem{kierstead2016choice}
Hal~A Kierstead, Andrew Salmon, and Ran Wang.
\newblock On the choice number of complete multipartite graphs with part size
  four.
\newblock {\em European Journal of Combinatorics}, 58:1--16, 2016.

\bibitem{kozik2014towards}
Jakub Kozik, Piotr Micek, and Xuding Zhu.
\newblock Towards an on-line version of ohba’s conjecture.
\newblock {\em European Journal of Combinatorics}, 36:110--121, 2014.

\bibitem{Kronk}
Hudson~V. Kronk and John Mitchem.
\newblock Critical point-arboritic graphs.
\newblock {\em J. London Math. Soc. (2)}, 9:459--466, 1974/75.

\bibitem{mudrock2018list}
Jeffrey~Allen Mudrock.
\newblock {\em On the list coloring problem and its equitable variants}.
\newblock PhD thesis, Illinois Institute of Technology, 2018.

\bibitem{noel2015proof}
Jonathan~A Noel, Bruce~A Reed, and Hehui Wu.
\newblock A proof of a conjecture of ohba.
\newblock {\em Journal of Graph Theory}, 79(2):86--102, 2015.

\bibitem{ohba2002chromatic}
Kyoji Ohba.
\newblock On chromatic-choosable graphs.
\newblock {\em Journal of Graph Theory}, 40(2):130--135, 2002.

\bibitem{rowshan2020catlin}
Yaser Rowshan and Ali Taherkhani.
\newblock A catlin-type theorem for graph partitioning avoiding prescribed
  subgraphs.
\newblock {\em arXiv preprint arXiv:2002.04702}, 2020.

\bibitem{xu2018list}
Rongxing Xu, Yeong-Nan Yeh, and Xuding Zhu.
\newblock List colouring of graphs and generalized dyck paths.
\newblock {\em Discrete Mathematics}, 341(3):810--819, 2018.

\bibitem{xue2012list}
Nini Xue and Baoyindureng Wu.
\newblock List point arboricity of graphs.
\newblock {\em Discrete Mathematics, Algorithms and Applications},
  4(02):1250027, 2012.

\bibitem{zhen2009list}
Lingyan Zhen and Baoyindureng Wu.
\newblock List point arboricity of dense graphs.
\newblock {\em Graphs and Combinatorics}, 25(1):123--128, 2009.

\bibitem{zhu2021chromatic}
Jialu Zhu and Xuding Zhu.
\newblock Chromatic $\lambda$-choosable and $\lambda$-paintable graphs.
\newblock {\em Journal of Graph Theory}, 98(4):642--652, 2021.

\end{thebibliography}
\end{document}